\documentclass[a4paper,UKenglish]{lipics-v2018}

\usepackage{microtype}

\title{Diagonal asymptotics for symmetric rational functions via ACSV}

\author{Yuliy Baryshnikov}{University of Illinois, Department of Mathematics\\{273 Altgeld Hall 1409 W. Green Street (MC-382), Urbana, IL 61801, USA}}{ymb@illinois.edu}{}{Partially supported by NSF grant DMS-1622370}

\author{Stephen Melczer}{University of Pennsylvania, Department of Mathematics\\{209 South 33rd Street, Philadelphia, PA 19104, USA}}{smelczer@sas.upenn.edu}{https://orcid.org/0000-0002-0995-3444}{Partially supported by an NSERC postdoctoral fellowship and NSF grant DMS-1612674}

\author{Robin Pemantle}{University of Pennsylvania, Department of Mathematics\\{209 South 33rd Street, Philadelphia, PA 19104, USA}}{pemantle@math.upenn.edu}{}{Partially supported by NSF grant DMS-1612674}

\author{Armin Straub}{University of South Alabama, Department of Mathematics and Statistics\\{411 University Blvd N, MSPB 325, Mobile, AL 36688, USA}}{straub@southalabama.edu}{https://orcid.org/0000-0001-6802-6053}{Partially supported by a Simons Collaboration Grant}

\authorrunning{Y. Baryshnikov, S. Melczer, R. Pemantle and A. Straub}

\Copyright{Yuliy Baryshnikov, Stephen Melczer, Robin Pemantle and Armin Straub}

\subjclass{Mathematics of computing $\rightarrow$ Combinatorics}


\category{}

\relatedversion{}

\supplement{}

\funding{}

\acknowledgements{Thanks to the Erwin Schr{\"o}dinger Institute, at which this work
was begun.  Thanks also to Petter Br{\"a}nd{\'e}n for help with 
Lemma~\ref{lem:GWS}.  }

\EventEditors{James Allen Fill and Mark Daniel Ward}
\EventNoEds{2}
\EventLongTitle{29th International Conference on Probabilistic, Combinatorial and Asymptotic Methods for the Analysis of Algorithms (AofA 2018)}
\EventShortTitle{AofA 2018}
\EventAcronym{AofA}
\EventYear{2018}
\EventDate{June 25--29, 2018}
\EventLocation{Uppsala, Sweden}
\EventLogo{}
\SeriesVolume{110}
\ArticleNo{12} 
\nolinenumbers 
\hideLIPIcs  

\usepackage[dvipsnames]{xcolor}
\usepackage{tikz}

\def\Z{\mathbb{Z}}
\def\R{\mathbb{R}}
\def\RP{\mathbb{RP}}
\def\C{\mathbb{C}}
\def\CC{{\cal C}}
\def\ee{\varepsilon}

\def\|{{\, | \, }}
\def\xx{{\bf x}}
\def\yy{{\bf y}}
\def\rr{{\bf r}}
\def\zz{{\bf z}}

\def\ww{{\bf w}}

\def\fd{\delta^Q}

\def\M{{\cal M}}
\def\Md{{\cal M}_d}
\def\cD{{\cal D}}
\def\D{{\cal D}}
\def\curv{{\mathcal K}}
\def\nbd{{\mathcal N}}
\def\grad{{\nabla}}
\def\diag{\textsf{diag}}
\def\crit{\textsf{crit}}
\def\zero{{\bf 0}}
\def\sing{{\mathcal V}}
\def\TT{{\bf T}}
\def\rhat{\hat{\rr}}
\def\hess{{\mathcal H}}

\newcommand{\bbeta}{\ensuremath{\boldsymbol \beta}}

\theoremstyle{plain}
\newtheorem{prop}[theorem]{Proposition}
\newtheorem{conj}[theorem]{Conjecture}
\newtheorem{question}[theorem]{Question}
\newtheorem*{unremark}{Remark}

\begin{document}

\maketitle

\begin{abstract}
We consider asymptotics of power series coefficients of rational 
functions of the form $1/Q$ where $Q$ is a symmetric multilinear 
polynomial.  We review a number of such cases from the literature, 
chiefly concerned either with positivity of coefficients or diagonal 
asymptotics.  We then analyze coefficient asymptotics using ACSV 
(Analytic Combinatorics in Several Variables) methods.  While ACSV 
sometimes requires considerable overhead and geometric computation, 
in the case of symmetric multilinear rational functions there are 
some reductions that streamline the analysis.  Our results include 
diagonal asymptotics across entire classes of functions, for example 
the general 3-variable case and the Gillis-Reznick-Zeilberger (GRZ) 
case, where the denominator in terms of elementary 
symmetric functions is $1 - e_1 + c e_d$ in any number $d$ of variables.  
The ACSV analysis also explains a discontinuous drop in exponential 
growth rate for the GRZ class at the parameter value $c = (d-1)^{d-1}$, 
previously observed for $d=4$ only by separately computing diagonal 
recurrences for critical and noncritical values of $c$.  
 \end{abstract}

\section{Introduction}

We study the power series coefficients of rational functions 
of the form $F(x_1, \dots , x_d) = 1/Q(x_1, \ldots , x_d)$ 
where $Q$ is a symmetric multilinear function with $Q(\zero) \neq 0$.  
Let 
$$ F(\xx) = \frac{1}{Q(\xx)} = \sum_{\rr\in\mathbb{Z}^d} a_\rr \xx^\rr, $$ 
converging in some polydisk $\cD \subset \mathbb{C}^d$.  Often one focuses 
on the diagonal coefficients $\delta_n := a_{n, \ldots , n}$, whose 
univariate generating function $\diag_F (z) := \sum_n \delta_n z^n$ 
satisfies a linear differential equation with polynomial coefficients, 
but may be transcendental.  A number of questions are natural, including 
nonnegativity (are all coefficients nonnegative), eventual nonnegativity 
(all but finitely many coefficients nonnegative), diagonal extraction 
(computing $\diag_F$ from $Q$), diagonal asymptotics, multivariate 
asymptotics and phase transitions in the asymptotics of $\{ a_\rr \}$. 

The positivity (nonnegativity) question is the most classical, dating back 
at least to Szeg{\H o}'s work in~\cite{szego-pos33}.  The techniques, some 
of which are indicated in the next section, used in the literature are diverse
and include integral methods and special functions, positivity preserving
operators, combinatorial identities, computer algebra such as cylindrical
algebraic decomposition, or determinantal methods.
Contrasting to these methods are analytic combinatorial several-variable 
methods (ACSV) as developed in~\cite{PW-book}. These are typically 
asymptotic, rather than exact, and therefore less useful for proving 
classical positivity statements, though they can be used to disprove them.  
Their chief advantages are their broad applicability and, increasingly, 
the level to which they have been automated. Our aim in this paper is 
to apply ACSV methods to a number of previously studied families of 
rational coefficient sequences, thereby extending what is known as well 
as illuminating the relative advantages of each method.

\subsection{Previously studied instances}

Let $\Md$ denote the class of symmetric functions of $d$ variables that are 
multilinear (degree~1 in each variable).  This class of generating functions 
$F(\xx) := 1/Q(\xx)$ where $Q \in \Md$ includes a great number of previously 
studied cases, some of which we now review.  Here and in the following, 
we use $d$ for the 
number of variables and boldface $\xx, \yy, \zz$, etc., for vectors of 
length $d$ of integer, real or complex numbers.  When $d$ is small we use 
$x, y, z, w$ for $x_1, x_2, x_3, x_4$. Let $e_k = e_{k,d}$ denote the $k^{th}$ 
elementary symmetric function of $d$ variables, the sum of all distinct $k$ element
products from the set of $d$ variables.  An equivalent description of 
the class $\Md$ is that it contains all linear combinations of 
$\{ e_{k,d} : 0 \leq k \leq d \}$.

The Askey-Gasper rational function is
\begin{equation} \label{eq:askey-gasper}
A(x,y,z) := \frac{1}{1 - x - y - z + 4xyz},
\end{equation}
which, in the previous notation, is $A(\xx) = F(\xx)$ when $d=3$ 
and $Q = 1 - e_1 + 4 e_3$.
Gillis, Reznick and Zeilberger \cite{zb-pos-el83} deduce positivity of 
$A$ from positivity of a $4$-variate extension due to 
Koornwinder~\cite{koornwinder-pos78}, for which they give a short 
elementary proof using a positivity preserving operation.  Gillis, 
Reznick and Zeilberger also provide an elementary proof of the 
stronger result by Askey and Gasper \cite{askey-pos77} that 
$A^\beta$ is positive for $\beta \geq ( \sqrt{17} - 3 ) / 2 \approx 0.56$, 
by deriving a recurrence relation for the coefficients that makes positivity 
apparent.

Specific functions in $\M_4$ that have shown up in the literature include 
the Szeg{\H o} rational function
\begin{equation} \label{eq:szego}
S(x,y,z,w) := \frac{1}{e_3(1-x,1-y,1-z,1-w)}
\end{equation}
as well as the Lewy-Askey function
\begin{equation} \label{eq:lewy}
L(x,y,z,w) := \frac{1}{e_2(1-x,1-y,1-z,1-w)},
\end{equation}
which is a rescaled version of $1/Q(\xx)$ with $d=4$ and 
$Q = 1 - e_1 + \frac23 e_2$.
Szeg{\H o} \cite{szego-pos33} proved that \eqref{eq:szego} is positive. 
In fact, he showed that $e_{d - 1, d}^{- \beta} (1 - \xx)$ is nonnegative 
if $\beta \geq 1 / 2$.  His proof relates the power series coefficients to 
integrals of products of Bessel functions and, among other ingredients, 
employs the Gegenbauer--Sonine addition theorem.  Scott and 
Sokal~\cite{ss-pos13} establish a vast and powerful generalization of 
this result by showing that, if $T_G$ is the spanning-tree polynomial 
of a connected series-parallel graph, then $T_G^{- \beta} (1 - \xx)$ 
is nonnegative if $\beta \geq 1 / 2$.  In the simplest non-trivial case, 
if $G$ is a $d$-cycle, then $T_G = e_{d - 1, d}$, thus recovering 
Szeg{\H o}'s result.  Relaxing the condition on $\beta$, Scott and 
Sokal further extend their results to spanning-tree polynomials of 
general connected graphs.  They do so by realizing that Kirchhoff's 
matrix-tree theorem implies that these polynomials can be expressed 
as determinants, and by proving that determinants of this kind are 
nonnegative.  As another consequence of this determinantal nonnegativity, 
Scott and Sokal conclude that~\eqref{eq:lewy} is nonnegative, 
thus answering a question originating with Lewy~\cite{askey-pos72} 
(with positivity replaced by nonnegativity).
Kauers and Zeilberger~\cite{kz-pos08} show that positivity of the 
Lewy-Askey rational function~\eqref{eq:lewy} would follow from 
positivity of the four variable function
\begin{equation} \label{eq:kz}
K(x,y,z,w) := \frac{1}{1 - e_1 + 2 e_3 + 4 e_4}.
\end{equation}
However, the conjectured positivity (or even nonnegativity) 
of~\eqref{eq:kz} remains open.

As noted above, $e_{d - 1, d}^{- \beta} (1 - \xx)$ is nonnegative if $\beta
\geq 1 / 2$.  The asymptotics of $e_{k,d}^{-\beta}(1 - \xx)$ are computed
in~\cite{BP-cones} for $(k,d)=(2,3)$.  In the cone $2(rs + rt + st) > 
r^2 + s^2 + t^2$, the coefficient $a_{r,s,t}$ is asymptotically positive 
when $\beta > 1/2 = (d-k)/2$ and not when $\beta < 1/2$.  
A conjecture of Scott and Sokal that remains open in both directions 
is that, for general $k$ and $d$, the condition $\beta \geq (d-k)/2$ 
is necessary and sufficient for nonnegativity of the coefficients of 
$e_{k,d}^{-\beta}(1 - \xx)$. 

Gillis, Reznick and Zeilberger~\cite{zb-pos-el83} consider the family
\begin{equation} \label{eq:GRZ}
F_{c,d} (x_1 , \ldots , x_d) := \frac{1}{1 - e_1 + c \, e_d} 
\end{equation}
of rational functions, where $c$ is a real parameter.
When $c < 0$, the coefficients are trivially positive, therefore it is usual 
to assume $c > 0$.  Gillis, Reznick and Zeilberger show that $F_{c,3}$ has 
nonnegative coefficients if $c \leq 4$ (and this condition is shown to be 
necessary in \cite{straub-pos}), but they conjecture that the threshold 
for $d \geq 4$ has a different form, namely that $F_{c,d}$ has 
nonnegative coefficients if and only if $c \leq d!$.  
It is claimed in~\cite{zb-pos-el83}, but the proof is omitted 
due to its length, that nonnegativity of $F_{d!, d}$ is 
implied by nonnegativity of the diagonal power series coefficients. 
In the cases $d = 4, 5, 6$, Kauers~\cite{kauers-pos07} proved 
nonnegativity of these diagonal coefficients by applying cylindrical 
algebraic decomposition (CAD) to the respective recurrences.  On the 
other hand, it is suggested in~\cite{sz-pos} that the diagonal 
coefficients are eventually positive if $c < (d - 1)^{d - 1}$. 

\subsection{Previous questions and results on diagonals}

The diagonal generating function $\diag_F$ and the sequence 
$\delta_n := a_{n, \ldots , n}$ it generates have received 
special attention.  One reason is that the question of multivariate 
asymptotics in the diagonal direction is simply stated, whereas the 
question of asymptotics in all possible directions requires discussion 
of different possible phase regimes, a notion of uniformity over 
directions, degeneracies when the coordinates are not of comparable 
magnitudes, and so forth.  Another reason is that there are effective 
methods for determining $\diag_F$ from $Q$, transferring the problem 
to the familiar univariate realm.

We briefly recall the theory of diagonal extraction. A $d$-variate power 
series $F$ is said to be D-finite if the formal derivatives 
$\{ \partial_\rr F : \rr \in (\Z^+)^d \}$ form a finite 
dimensional vector space over $\C [\xx]$.  In one variable, this is 
equivalent to $F$ satisfying a linear differential equation with 
polynomial coefficients,
$$\sum_{i=0}^k q_i (z) \frac{d^i}{dz^i} F = 0, \quad q_i \in \mathbb{C}[z] .$$

\begin{prop}[D-finite closure under diagonals~\cite{lipshitz-diagonal}]
\label{pr:lipshitz}
Let $F(\xx)$ be a D-finite power series.  Then $\diag (z) := \sum_n 
\delta_n z^n$ is D-finite, where $\delta_n := a_{n, \ldots , n}$. 
\end{prop}

When $F$ is a rational function and $d=2$, it was known that $\diag$ 
is algebraic (and thus D-finite) at least by the late 
1960's~\cite{Furstenberg1967,hautus-klarner-diagonal}, 
and in special cases by P{\'o}lya in the 1920's~\cite{Polya1921}.   
In the rational function $F(x,y) = P(x,y)/Q(x,y)$ one substitutes 
$y=1/x$ and computes a residue integral to extract the constant coefficient.  
The basis for Lipshitz' proof was the realization that the complex 
integration can be viewed as purely formal.  With the advent of 
computer algebra this formal D-module computation was automated, 
with an early package in Macaulay and more widely used modern 
implementations in Magma, Mathematica and Maple.  Due to advances 
in software and processor speed, these computations are often completable 
on functions arising in applications. Christol~\cite{Christol1984} was the first to show
that diagonals of \emph{rational} functions are D-finite.

The following relationship between D-finiteness of a univariate function 
and the existence of a polynomial recursion satisfied by its coefficient 
sequence is the result of translating a formal differential equation 
into a relation among the coefficients.
\begin{prop} \label{pr:P-recursion}
The series $f(z) = \sum_{n \geq 0} a_n z^n$ is D-finite if and only 
if it is polynomially recursive, meaning that there is a $k > 0$ and 
there are polynomials $p_0, \ldots, p_k$, not all zero, such that for 
all but finitely many $n$, 
$$\sum_{i=0}^k p_i (n) f(n+i) = 0 \, .$$
\end{prop}

Let $f$ be a D-finite power series in one variable.  If $f$ has positive 
finite radius of convergence and integer coefficients, then it is a 
so-called {\em G-function} and has well behaved asymptotics according 
to following result.
\begin{prop}[Asymptotics of G-Function Coefficients] \label{pr:G}
Suppose $f$ is D-finite with finite radius of convergence and integer 
coefficients annihilated by a minimal order linear differential operator 
$\mathcal{L}$ with polynomial coefficients.  Then $\mathcal{L}$ has only 
regular singular points in the Frobenius sense.  Consequently, 
the coefficients $\{ a_n \}$ are given asymptotically by a formula 
\begin{equation} \label{eq:regular}
a_n \sim \sum_\alpha C_\alpha n^{\beta_\alpha} \rho_\alpha^{-n} 
   (\log n)^{k_\alpha}
\end{equation}
where the sum is over quadruples $(C_\alpha, b_\alpha, \rho_\alpha, k_\alpha)$ 
as $\alpha$ ranges over a finite set $A$ with the following properties.  
The base $\rho_\alpha$ is an algebraic number, a root of the leading 
polynomial coefficient of $\mathcal{L}$.  The $\beta_\alpha$ are 
rational and for each value of $\rho_\alpha$ can be determined as roots 
of an explicit polynomial constructed from $\rho_\alpha$ and $\mathcal{L}$.  
The log powers $k_\alpha$ are nonnegative integers, zero unless for fixed 
$\rho_\alpha$ there exist two values of $\beta_\alpha$ differing by an 
integer (including multiplicities in the construction of $\beta_\alpha$).  
The $C_\alpha$ are not in general closed form analytic expressions, but 
may be determined rigorously to any desired accuracy. 
\end{prop}

\begin{proof} 
The discussion in~\cite[page~37]{melczer-phd} 
gives references to several published results that together establish
this proposition; see also Flajolet and 
Sedgewick~\cite[Section VII. 9]{flajolet-sedgewick-anacomb}.  
Determination of all rational and algebraic numbers 
other than $C_\alpha$ is known to be effective.
\end{proof}

Because there are computational methods for the study of diagonals, 
it is of interest to reduce positivity questions to those involving
only diagonals.  For the Gillis-Reznick-Zeilberger class $F_{c,d}$, 
such a result is conjectured.

\begin{conj}[\cite{zb-pos-el83}] \label{conj:diagonal}
For $d \geq 4$, the following three statements are equivalent.
\begin{enumerate}[(i)]
\item $c \leq d!$
\item The diagonal coefficients of $F_{c,d}$ are nonnegative
\item All coefficients of $F_{c,d}$ are nonnegative
\end{enumerate}
\end{conj}

To be precise, $(iii) \Rightarrow (ii) \Rightarrow (i)$ is trivial
(look at $\delta_1$); nonnegativity of all coefficients of $F_{c,d}$ 
holds for some interval $c \in [0 , c_{\max}]$, therefore the conjecture
comes down to nonnegativity of $F_{d! , d}$. A proof for $(ii) 
\Rightarrow (iii)$ in the case $c = d!$ is claimed in~\cite{zb-pos-el83} 
but omitted from the paper due to length.
This question is generalized in~\cite{sz-pos} to all of $\Md$.

\begin{question}[\protect{\cite[Question 1.1 and following]{sz-pos}}]
For $Q \in \Md$ and $F = 1/Q$, under what conditions does
nonnegativity of the coefficients of $\diag_F$ imply nonnegativity
of all coefficients of $F$?
\end{question}

More specifically, with nonnegativity in place of positivity, the authors 
of that paper wonder whether positivity of $F$ is equivalent to positivity 
of $\diag_F$ together with positivity of $F(x_1, \ldots , x_{d-1} , 0)$.
They prove that this is true for $d=2$ and, with additional evidence, 
conjecture this to be true for $d=3$ as well.  Combined 
with~\cite[Conjecture~1]{straub-pos} and \cite[Conjecture 3.3]{sz-pos}, 
we obtain the following explicit predictions on the diagonal coefficients.

\begin{conj} \label{conj:M3}
Let $F = 1/Q$ where $Q = 1 - e_1 + a e_2 + b e_3$, which is, up to rescaling,
the general element of $\M_3$.
Then $\diag_F$ is nonnegative if and only if
\begin{equation}
b \leq \begin{cases}
   6(1-a) & a \leq a_0 \\ 2-3a+2(1-a)^{3/2} & a_0 \leq a \leq 1 \\ 
   -a^3 & a \geq 1, \end{cases}
\end{equation}
where $a_0 \approx -1.81$ is characterized by $6(1-a_0) = 2-3a_0+2(1-a_0)^{3/2}$.
\end{conj}

\subsection{Present results}

In the present work we use ACSV to answer asymptotic versions of 
these questions.  Aside from computing special cases, the main 
new results are (1) simplification for diagonals with symmetric denominators
via the Grace-Walsh-Szeg{\H o} Theorem (Lemma~\ref{lem:GWS} below); 
(2) an easy further simplification for the Gillis-Reznick-Zeilberger 
class (Lemma~\ref{lem:Fcd} below); and (3) a topological computation 
to explain the drop in magnitude of coefficients at critical parameter values
(Theorem~\ref{th:lacuna} below).

The first special case we look at is the diagonal of the general 
element of $\M_3$, corresponding to Conjecture~\ref{conj:M3}.
\begin{theorem} \label{th:diag M3}
Let $Q = 1 - e_1 + a e_2 + b e_3$, let $F = 1/Q = \sum_\rr a_\rr \zz^{\rr}$
and let $\delta_n = a_{n, \dots , n}$ be the diagonal coefficients of $F$.
Then $\delta_n$ is eventually positive when 
\begin{equation}
b < \begin{cases}
   -9a & a \leq -3 \\ 2 - 3a + 2(1-a)^{3/2} & -3 \leq a \leq 1 \\ 
   - a^3 & a \geq 1 
   \end{cases}
\label{eq:M3cases}
\end{equation}
while, when the inequality is reversed, $\delta_n$ attains an infinite
number of positive and negative values.
\end{theorem}

Theorem~\ref{th:diag M3} is obtained by examining asymptotic regimes, 
captured in the following result.

\begin{theorem} \label{th:asm M3}
Let $Q , F,$ and $\delta_n$ be as in Theorem~\ref{th:diag M3}.
Assuming that $b$ is not equal to the piecewise function in 
Equation~\eqref{eq:M3cases},
\begin{equation} 
\delta_n = \sum_{x \in E}\left(\frac{x^{-3n}}{n} \cdot 
   \left| \frac{1-2ax-bx^2}{1-ax} \right| \cdot 
   \frac{1}{2\sqrt{3}(1-2x+ax^2)} \right) 
   \left(1+O\left(\frac{1}{n}\right)\right), \label{eq:asm M3}
\end{equation}
where $E$ consists of the minimal modulus roots of the polynomial 
$Q(x,x,x)=1 - 3x + 3a x^2 + b x^3$.
\end{theorem}

The situation for eventual positivity on the diagonal when equality holds 
in Equation~\eqref{eq:M3cases} is more delicate.  When $a < -3$ 
it follows from seeing that there are two diagonal minimal points, $(r,r,r)$ 
and $(-r,-r,-r)$, with a greater constant at the positive point.
When $-3 < a < 1$, it follows from a dominant positive real cone point.
When $a = -3$ a quadratically degenerate smooth point at $(-1/3,-1/3,-1/3)$
may be shown via rigorous numerical diagonal extraction to dominate 
the cone point at $(1/3,1/3,1/3)$, leading to alternation.  
When $a = 1$, $a_\rr \equiv 1$.  Finally, when $a > 1$, there are 
three smooth points on the unit circle, with nonnegativity conjectured 
because the positive real point is degenerate and should dominate. 

Our second set of results concern the diagonal of the
general element of the GRZ rational function $F_{c,d}$.
Let 
\begin{equation} \label{eq:c_*}
c_* = c_* (d) := (d-1)^{d-1} \, .
\end{equation}
The following corresponds to Conjecture~\ref{conj:diagonal}.
\begin{theorem} \label{th:asym pos}
Let $d \geq 4$.  Then the diagonal coefficients of $F_{c,d}$ are 
eventually positive when $c < c_*$ and contain an infinite number
of positive and negative values when $c > c_*$. When $c<c_*$, 
there is a conical neighborhood $\nbd$ of the diagonal
such that $a_\rr > 0$ for all but finitely many $\rr \in \nbd$.
\end{theorem}

Again, the result is obtained through an explicit asymptotic analysis.

\begin{theorem} \label{th:asm Fcd}
Let $\delta_n$ be the diagonal coefficients of $F_{c,d}$. Then when 
$c \neq c_*$,
\[ \delta_n = \sum_{x \in E}\left(\frac{x^{-dn}}{n^{(d-1)/2}} \cdot 
   \left(\frac{2\pi(1-(d-1)r)}{r^{(d-1)/2}}\right)^{(d-1)/2} \cdot
   \frac{1}{d^{1/2}(1-(d-1)r)}\right) \left(1+O\left(\frac{1}{n}
   \right)\right), \]
where $E$ consists of the minimal modulus roots of the polynomial 
$1/F_{c,d}(x,\dots,x)=1 - dx + c x^d$.
\end{theorem}

These theorems are proven in Section~\ref{sec:Fcd}, using ACSV smooth
point methods summarized in Section~\ref{sec:ACSV},  however the case 
$c = c_*$ for the GRZ rational function requires the more delicate 
results of Section~\ref{sec:lacuna}.

\subsection{Exponential drop and further results}

In the GRZ family, for even values of $d \geq 4$ the exponential growth 
rate of the coefficients drops at the special value $c = (d-1)^{d-1}$.  
This special value, and the corresponding drop in exponential growth, may 
be identified for each fixed $d$ from the differential equation annihilating 
the diagonal.  For example, when $d=4$ an annihilating differential equation 
for the diagonal of $F_{c,4}$ is computed by D-module integration in the Mathematica package
of Koutschan~\cite{Koutschan2010b}
producing the annihilating operator
$\mathcal{L}$, of order~3 and maximum coefficient degree~8, such that 
$\mathcal{L} \diag_{F_{c,4}} = 0$:
\begin{equation} 
\label{eq:Fc4}
\begin{split}
\mathcal{L} &= z^2(c^4z^4+4c^3z^3+6c^2z^2+4cz-256z+1)(3cz-1)^2
  \partial_z^3 \\
& + 3z(3cz-1)(6c^5z^5+15c^4z^4+8c^3z^3-6c^2z^2-384cz^2-6cz+384z-1)
   \partial_z^2 \\
& + (cz+1)(63c^5z^5-3c^4z^4-66c^3z^3+18c^2z^2+720cz^2+19cz-816z+1)
   \partial_z \\
& + 9c^6z^5-3c^5z^4-6c^4z^3+18c^3z^2-360c^2z^2+13c^2z-384cz+c-24.
\end{split}
\end{equation}

When $c = 27$, all coefficients in~\eqref{eq:Fc4} acquire 
enough zeros at $z = 1/81$ that the quantity $(81 z - 1)^4$
may be factored out of the entire operator, leaving the following
operator of order~3 and maximum degree~4:
\begin{equation} 
\begin{split}
{\mathcal L}_{27} := &z^2 (81 z^2 + 14 z + 1) \; \partial_z^3 
+ 3 z (162 z^2 + 21 z + 1) \; \partial_z^2 \\
&+ (21 z + 1) (27 z + 1) \partial_z + 3 (27 z + 1).
\end{split}
\end{equation}

Asymptotics for $\delta_n$ may be extracted via the methodology
described in Proposition~\ref{pr:G}.  In the special case $d=4, c=27$, 
the recursion may be found on the OEIS (entry A125143) and 
identifies $\{ \delta_n \}$ as the 
{\em Almkvist--Zudilin numbers}\footnote{That these are the diagonals 
of the rational function $F_{27, 4}$ was observed in \cite{s-apery}, 
where it is further conjectured that the coefficients of $F_{27,4}$ 
satisfy very strong congruences.}
from~\cite[sequence (4.12)($\, \delta$)]{asz-clausen}.  The known
asymptotic formula implies that $|\delta_n|^{1/n} \to 9$.  However,
as $c \neq 27$ approaches 27 from either side, we have
$$\lim_{c \to 27} \lim_{n \to \infty} |\delta_n|^{1/n} = 81;$$
in other words, the growth rate at $c=27$ drops suddenly from~81 to~9.
The occurrence of a phase change at $(d-1)^{d-1}$ for all $d$ 
and drop in exponential rate for even $d \geq 4$ had not 
previously been proved.  The special role of the case $c=(d-1)^{d-1}$ was
observed in~\cite[Example~4.4]{sz-pos} and claimed to agree with intuition
from hypergeometric functions.  
We verify this, first by identifying the singularity from an ACSV point 
of view and then by checking that this singularity indeed produces the 
observed dimension drop.

\begin{theorem}[exponential growth approaching criticality] \label{th:d-1}
For all $d \geq 2$,
$$ \lim_{c \to c_*} \limsup_{n \to \infty} |\delta_n|^{1/(dn)} = d-1 \, .$$
\end{theorem}

\begin{theorem}[dimension drop at criticality] \label{th:d-1 critical}
When $c = c_*$ and $d\geq4$ is even, 
$$\limsup_{n \to \infty} |\delta_n|^{1/(dn)} < d-1 \, .$$
\end{theorem}

Theorem~\ref{th:d-1 critical} is proved in Section~\ref{sec:lacuna}.

\section{ACSV} \label{sec:ACSV}

In this section we describe the basic setup for ACSV and state some
existing results.  Definitions for the topological and geometric 
quantities used below can be found in Pemantle and Wilson~\cite{PW-book}.
Throughout this section let $F(\zz) = P(\zz)/Q(\zz)
= \sum_\rr a_\rr \zz^\rr$ denote a rational series in $d$ variables, 
with $P$ and $Q$ co-prime polynomials.
Assume that $F$ has a (finite) positive radius of convergence; that is, 
$Q(\zero) \neq 0$ and $P/Q$ is not a polynomial.  Let $\sing := 
\{\zz \in \C^d : Q(\zz) = 0\}$ denote the singular variety for $F$ and let 
$\M = (\C^*)^d \setminus \sing$ where $\C^* = \C \setminus \{ 0 \}$.  
Coefficients $a_\rr$ are extracted via the multivariate Cauchy formula
\begin{equation} \label{eq:cauchy}
a_\rr = \frac{1}{(2 \pi i)^d} \int_{\TT} \zz^{-\rr} F(\zz) \frac{d\zz}{\zz},
\end{equation}
where $d\zz / \zz$ denotes the holomorphic logarithmic volume form
$(dz_1 / z_1) \wedge \cdots \wedge (dz_d/z_d)$ and $\TT$ denotes a
small torus (a product of sufficiently small circles about
the origin in each coordinate, so that the product of the corresponding
disks is disjoint from $\sing$).  The fundamental insight of ACSV is
that the integral depends only on the homology class of $\TT$ in 
$H_d(\M)$.  Therefore, one tries to replace $\TT$ by some homologous
chain $\CC$ over which the integral is easier, typically via some 
combination of residue reductions and saddle point estimates.

A {\em direction} of asymptotics is an element $\rhat \in (\RP^d)^+$;
that is, a projective vector in the positive orthant. If $\rr \in (\R^d)^+$
we write $\rhat$ to denote the representative $\rr/|\rr|$ of the
projective equivalence class containing 
$\rr$, where $|\rr|=|\rr|_1 := r_1+\cdots+r_d$.
Given a Whitney stratification of $\sing$ into smooth manifolds, the 
{\em critical
set} $\crit(\rhat)$ for a direction $\rhat$ is the set of $\zz \in \sing$ 
such that $\rhat$ is orthogonal to the tangent space of the stratum of 
$\zz$ in $\sing$.  If $\zz$ is a smooth point of $\sing$ and $Q$ is 
square-free, this means $\rhat$ should be parallel to the logarithmic gradient 
$(z_1 \partial Q / \partial z_1, \ldots , z_d \partial Q / \partial z_d)$.
A {\em minimal} point for direction $\rhat$ is a point $\zz \in \crit (\rhat)$
such that the open polydisk $\D (\zz) := \{ \ww : |w_j| < |z_j| \, 
\forall 1 \leq j \leq d \}$ does not intersect $\sing$.  The minimal 
point $\zz$ is called {\em strictly minimal} if the closed polydisk
$\overline{\D (\zz)}$ intersects $\sing$ only at $\zz$.  

For any $\bbeta \in \R^d$,
let $\TT (\bbeta) = \{ \ww : |w_j| = \exp (\beta_j) \,\, \forall \, 1 \leq 
j \leq d \}$ denote the torus of points with log modulus vector $\bbeta$.
The \emph{amoeba} of $Q(\zz)$ is the image of $\sing$ under the map
$\text{Relog}(\zz) = (\log|z_1|,\dots,\log|z_d|)$, while the
\emph{height} of a point $\zz$ is $h_{\rhat}(\zz) = -\rhat \cdot \text{Relog}(\zz)$.
Except in Section~\ref{sec:lacuna}, 
all ACSV computations are based on the following result.
\begin{theorem}[smooth point formula] \label{th:smooth}
Fix $F = P/Q = \sum_\rr a_{\rr} \zz^\rr$ and vector $\rr \in (\R^d)^+$
in direction $\rhat$.  Assume there exists $\bbeta \in \R^d$ such that 
the following two hypotheses hold.

\begin{enumerate}[{\bf 1}]
\item {\bf Finite critical points on the torus.}  \label{i:smooth}
The set $E : = \TT(\bbeta) \cap \crit(\rhat)$ is finite, nonempty 
and contains only minimal smooth points. 
\item {\bf Quadratic nondegeneracy}.
At each $\zz \in E$ fix $k=k(\zz)$ such $\partial Q / \partial z_k (\zz) \neq 0$ 
and let $z_k = g(z_1,\dots,\hat{z_k},\dots,z_d)$ be a smooth local 
parametrization of $z_k$ on $\sing$ as a function of $\{ z_j : j \neq k \}$. 
We assume that the Hessian determinant $\hess_{k(\zz)}$ of second partial 
derivatives of $g\left(w_1e^{i\theta_1},\dots,w_de^{i\theta_d}\right)$ 
with respect to the $\theta_j$ at the origin is non-zero for each 
$\zz \in E$. 
\end{enumerate}

Then there exists a closed neighborhood $\nbd$ of $\rhat$ in $(\R^d)^+$ on which all
the above hypotheses hold and, for any $\rr$ with $\rhat$
in this neighborhood, 
%
\begin{equation} \label{eq:smooth}
a_\rr = ( 2 \pi)^{(1-d)/2} \sum_{\zz \in E} 
   \det \hess_{k(\zz)}^{-1/2} 
   \frac{P(\zz)}{z_k (\partial Q / \partial z_k)(\zz)} r_k^{(1-d)/2} \zz^{-\rr}
   + O \left ( r_k^{-d/2} \zz^{-\rr} \right ) \, .
\end{equation}  
\end{theorem}

\begin{unremark}
A number of other formulae for $a_\rr$ are equivalent to this one
and hold under the same hypotheses.
An explicit formula for $\hess_k$ in terms of partial derivatives 
of $Q$ is given in~\cite[Theorem 54]{melczer-phd}.  The following 
coordinate-free formula for the constants involved in terms of the 
complexified Gaussian curvature $\curv$ at a smooth point $\zz \in \sing$ 
is given in~\cite[(9.5.2)]{PW-book} as
\begin{equation} 
\label{eq:curv}
a_\rr = (2 \pi)^{(1-d)/2} \left [ \sum_{\zz \in E} 
   \curv (\zz)^{-1/2} \; |\grad_{\log} Q(\zz)|^{-1} P(\zz) \; 
   |\rr|^{(1-d)/2} \, \zz^{-\rr} \right ] 
+ O \left ( |\rr|^{-d/2} |\zz|^{-\rr} \right )
\end{equation}
\end{unremark}

\begin{proof}
Assume first that $\log |\ww|$ is
the unique minimizer of $\rr \cdot \xx$ on the boundary of the log domain
of convergence (this being a component of the complement of the amoeba).
Under no assumptions on $E$ or $\curv$, Theorem~9.3.2 of~\cite{PW-book} 
writes the multivariate Cauchy integral~\ref{eq:cauchy} as the integral 
of a residue form $\omega$ over an intersection cycle, $\CC$.  Taking into 
account that $E$ is finite, and assuming an extra hypothesis that $\rr$ is
a {\em proper direction} (see~\cite[Definition~2.3]{BP-cones}), 
Theorem~9.4.2 of~\cite{PW-book} 
identifies $\CC$ as a sum of quasi-local cycles near the points of $E$.  
For each such $\zz$, if $\partial Q / \partial z_k$ and $\det \hess_k$ 
do not vanish, Theorem~9.2.7 of~\cite{PW-book} identifies the integral 
as the corresponding summand in~\eqref{eq:smooth}.  
Nonvanishing of $\hess_k$ is equivalent to nonvanishing of $\curv$,
leading to the coordinate-free formula~\eqref{eq:curv}, which may be
found in~\cite[Theorem~9.3.7]{PW-book}.  This proves the theorem under
an extra hypothesis on the amoeba boundary.  

To remove the properness hypothesis, consider the intersection cycle $\CC$
obtained from expanding the torus $\TT(\bbeta - \epsilon\rr)$ inside the domain of convergence
of $F$ to a torus  $\TT(\bbeta + \epsilon\rr)$.
The construction in~\cite[Section~A4]{PW-book}
gives a compact $(d-1)$-chain representing a relative cycle in 
$H_{d-1} (\sing^{c+\ee} , \sing^{c-\ee})$; that is, a chain of
maximum height $c+\ee$ with maximum boundary height $c-\ee$.  
Applying the downward gradient flow of $h_{\rhat}$ on $\sing$ for arbitrarily small time, 
we arrive again at a chain satisfying the conclusions 
of~\cite[Theorem~9.4.2]{PW-book}.  Because the deformed chain 
has nonvanishing boundary, one must add a term for the chain swept
out by the deformation applied to this boundary, but the elements of this 
chain have height at most $c - \ee$ so the resulting integral will
be within the error term above. 
\end{proof}

\begin{corollary} \label{cor:asym pos}
Assume the hypotheses of Theorem~\ref{th:smooth}, and fix a vector 
$\mathbf{v}$ in direction $\rhat$.
\begin{enumerate}[(i)]
\item If $E = \{ \zz \}$ for some $\zz$ in the positive real orthant 
in $\C^d$ and the leading constant of Equation~\eqref{eq:smooth} is positive, 
then there exists a neighbourhood of $\rhat$ such that all but finitely 
many coefficients $\{ a_\rr : \rhat \in \nbd \}$ are positive. \label{item:1}
\item If $E = \{ \zz \}$ for some $\zz$ such that
$\zz^{\bf v} := \prod_{j=1}^d z_j^{v_j}$ is positive real and the leading 
constant of Equation~\eqref{eq:smooth} is positive, 
then all but finitely many coefficients $a_{n\bf v}$ are positive. 
\label{item:2}
\item If $E$ does not contain a point $\zz$ with $\zz^{\bf v}$
positive real and the sum in Equation~\eqref{eq:smooth} is not 
identically zero, then infinitely many coefficients $a_{n\bf v}$
are positive and infinitely many $a_{n\bf v}$ are negative. \label{item:3}
\end{enumerate}
\end{corollary}

\begin{unremark}
When $E$ contains a point in the positive real orthant but it is not 
a singleton, the corollary does not provide information as to eventual 
positivity.
\end{unremark}

\begin{proof} 
Conclusions~\eqref{item:1} and~\eqref{item:2} 
follow immediately from~\eqref{eq:smooth} because the sum is a single 
positive term. 

For conclusion~\eqref{item:3}, grouping the elements of $E$ by 
conjugate pairs we note that up to scaling by $\zz^{n \bf v}n^{d/2}$ 
the asymptotic leading term of $a_{n\bf v}$ has the form
\[ l_n = \sum_{i=1}^{|E|}a_i \cos(2\pi \theta_i n + \beta_i), \]
where each $\theta_i,a_i,\beta_i$ is real, and $\theta_i \in (0,1)$.
If $r_n$ is any sequence satisfying a linear recurrence 
relation with constant coefficients, and $r_n=O(1/n)$, 
then Bell and Gerhold~\cite[Section 3]{BellGerhold2007} show that 
$l_n > r_n$ infinitely often.  Since the modulus of the error term 
in Equation~\eqref{eq:smooth} can be bounded by a linear recurrence 
sequence with growth $O(1/n)$, we see that $a_{n\bf v}$ is positive 
infinitely often.  Repeating the argument with $-l_n$ shows that 
$a_{n\bf v}$ is negative infinitely often.
\end{proof}

Any computer algebra system can compute the set of smooth critical 
points in $\crit(\rhat)$ by
solving the $d-1$ equations $(\grad_{\log}Q) (\zz) \parallel \rhat$
together with the equation $Q(\zz) = 0$, where 
$\grad_{\log}Q = \left(z_1 \partial Q/\partial z_1, \dots, z_d \partial Q/\partial z_d\right)$.  
Identifying which points 
in $\crit$ are minimal is more difficult, although still 
effective~\cite{MelczerSalvy2016}.  For our cases, we can use results 
about symmetric functions to help with the computations.
For any polynomial $Q$ in $d$ variables, let $\delta^Q$ denote the 
codiagonal: the univariate polynomial defined by $\delta^Q (x) 
= Q(x, \ldots , x)$.  

\begin{lemma}[polynomials in $\Md$ have diagonal minimal points] \label{lem:GWS}
Let $F = 1/Q$ with $Q \in \Md$.  Let $x$ be a zero of $\delta^Q$ of 
minimal modulus.  Then $\xx := (x, \ldots , x)$ is a minimal point 
for $F$ in $\crit (1, \ldots , 1)$.
\end{lemma}

This follows directly from the classical Grace-Walsh-Szegő Theorem, 
a modern proof of which is contained in the following.

\begin{proof}
Let $\alpha_1 , \ldots , \alpha_k$ be the roots of $\delta^Q$, 
where $k \leq d$ is the common degree of $Q$ and $\delta^Q$ and 
$|\alpha_1|$ is minimal among $\{ |\alpha_j| : j \leq k \}$.
For any $\ee > 0$, the polynomial
$$M(\xx) := \prod_{j=1}^k (x_j - \alpha_j)$$
has no zeros in the polydisk $\D$ centered at the origin whose 
radii are $\alpha_1 - \ee$.  The symmetrization of $M$ 
(see~\cite{borcea-branden-LYPS2})
is defined to be the multilinear symmetric function $m$ such that
$m(x,\dots,x) = M(x,\dots,x)$.  In our case $M(x,\dots,x)= \delta^Q(x)$, and
it immediately follows that $m=Q$.  By the
Borcea-Br{\"a}nd{\'e}n symmetrization lemma 
(see~\cite[Theorem~2.1]{borcea-branden-LYPS2}), the polynomial
$Q$ has no zeros in the polydisk $\D$.  We conclude that the 
zero $\xx$ of $Q$ is a minimal point of $F$.  
\end{proof}

\section{Symmetric multilinear functions of three variables} \label{sec:M3}

In this section we determine the diagonal asymptotics for general
$Q = 1 - e_1 + a e_2 + b e_3 \in \M_3$.  Taking the coefficient of $e_1$
to be~1 loses no generality because of the rescaling $x_j \to \lambda x_j$
which preserves $\Md$ and affects coefficient asymptotics in a trivial way.
In order to use Theorem~\ref{th:smooth}, we begin by identifying 
minimal points.  Lemma~\ref{lem:GWS} dictates that our search should
be on the diagonal.  

To that end, let $\fd(x) = Q(x,x,x) = 1 - 3x + 3a x^2 + b x^3$.
The discriminant of $\fd$ is a positive real multiple of
$p(a,b) := 4a^3 - 3a^2 + 6ab + b^2 - 4b = (a-1 + 3(b-1))^2 - 4 (b-1)^3$, and
the zero set of $\fd$ is obtained from that of the cubic $4 b^3 = - a^2$
by centering at $(1,-1)$ and shearing via $(a,b) \mapsto (a+3b , b)$.  
The discriminant $p(a,b)$ vanishes along the red curve (solid and dashed)
in Figure~\ref{fig:1}.  Let $r_1 (a)$ and $r_2 (a)$ denote respectively 
the upper and lower branches of the solution to $p(a,b) = 0$.

\begin{figure}
\centering
\begin{tikzpicture}[x=0.8cm,y=0.08cm]
\draw[-,color=black] (-5,0) -- (4,0);
\foreach \x in {-4,-3,-2,-1,1,2,3}\draw[shift={(\x,0)},color=black] (0pt,2pt) -- (0pt,-2pt) node[below] {\footnotesize $\x$};
\draw[-,color=black] (0,-30) -- (0,45);
\foreach \y in {-30,-20,-10,10,20,30,40}\draw[shift={(0,\y)},color=black] (2pt,0pt) -- (-2pt,0pt) node[left] {\footnotesize $\y$};
\clip(-5,-30) rectangle (4,45);
\draw[line width=1pt,color=red,smooth,samples=100,domain=-5:1,opacity=0.5, dashed] plot(\x,{-3*(\x)+2-2*sqrt(-(\x)^(3)+3*(\x)^(2)-3*(\x)+1)});
\draw[line width=1pt,color=red,smooth,samples=100,domain=-5:-3,opacity=0.5, dashed] plot(\x,{-3*(\x)+2+2*sqrt(-(\x)^(3)+3*(\x)^(2)-3*(\x)+1)});
\draw[line width=1pt,color=red,smooth,samples=100,domain=-3:1,opacity=1] plot(\x,{-3*(\x)+2+2*sqrt(-(\x)^(3)+3*(\x)^(2)-3*(\x)+1)});

\draw[line width=1pt,color=Green,smooth,samples=100,domain=-5:1,opacity=0.4, dashed] plot(\x,{-(\x)^3});
\draw[line width=1pt,color=Green,smooth,samples=100,domain=1:4,opacity=1] plot(\x,{-(\x)^3});

\draw[line width=1pt,color=blue,smooth,samples=100,domain=-5:-3,opacity=1] plot(\x,{-9*(\x)});
\draw[line width=1pt,color=blue,smooth,samples=100,domain=-3:4,opacity=0.4, dashed] plot(\x,{-9*(\x)});
\end{tikzpicture}
\caption{The three regimes defined by Proposition~\ref{pr:1}, 
made up of the curves {\color{blue}$b=-9a$}, {\color{red}$p(a,b)=0$}, 
and {\color{Green}$b=-a^3$}. Dashed lines represent the curves where they 
do not determine positivity of coefficients; note smoothness in the
transitions between regimes.}
\label{fig:1}
\end{figure}
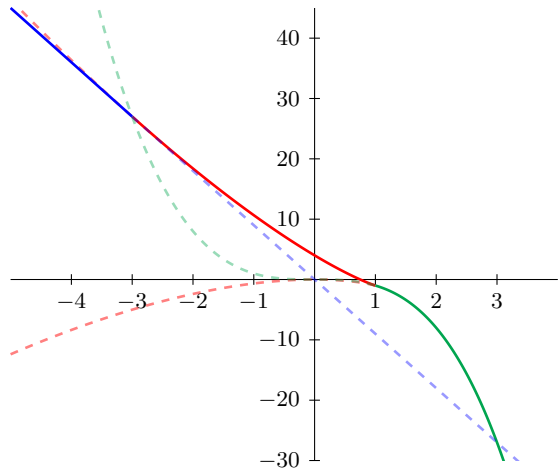

\begin{lemma}
Let $p$ be a minimal modulus root of $\fd$.  Then any critical point of $1/Q$ 
on the torus $T(p,p,p)$ has the form $(q,q,q)$ where $\fd(q)=0$.
\end{lemma}

\begin{proof}
Gr{\"o}bner basis computations show nondiagonal critical points to be
permutations of $\left(\frac{1}{a},\frac{1}{a},\frac{a(1-a)}{a^2+b}\right)$,
occurring when $b = a^2 (a-2)$.   When $a \leq 1$, the only time the positive 
root of $\fd(x)$ has modulus $1/|a|$ is the trivial case $(a,b)=(1,-1)$.  
When $b = a^2(a-2)$ and $a>1$, the modulus of the product of the roots 
of $\fd(x)$ equals $\frac{1}{a^2(a-2)}$ and the minimal 
roots of $\fd(x)$ are a pair of complex conjugates.  If this pair has 
modulus $1/a$, then the real root of $\fd(x)$ is $\pm \frac{1}{a^4(a-2)}$, 
but $\fd\left(\pm \frac{1}{a^4(a-2)}\right) \neq 0$ for $a>1$.
\end{proof}

Determining asymptotics is thus a matter of determining the minimal 
modulus roots of $\fd(x)$.  The following may be proved by comparing
moduli of roots, separating cases according to the sign of $p(a,b)$.

\begin{prop} \label{pr:1}
The function $\fd$ has a minimal positive real zero if and only if
$$b \leq \begin{cases}
   -9a & a \leq -3 \\ r_1(a) & -3 \leq a \leq 1 \\ - a^3 & a \geq 1
   \end{cases}
$$
This corresponds to the set of points lying on and below the solid curve 
in Figure~\ref{fig:1}.
\end{prop}

\begin{proof}[Proof of Theorems~\ref{th:diag M3} and~\ref{th:asm M3}:]
Suppose $b$ is greater than the piecewise expression in the
proposition; then $\delta^Q$ has no minimal positive zero,
so the product of the three coordinates of the minimal points
determined above do not lie in the positive orthant.
By part~\eqref{item:3} of Corollary~\ref{cor:asym pos}, the
diagonal coefficients are not eventually positive.
Asymptotics of $\delta_n$ are determined by Theorem~\ref{th:smooth}, 
and when $b$ is less than the piecewise expression it can be verified 
that the dominant term is positive.
\end{proof}

\section{The Gillis-Reznick-Zeilberger classes} \label{sec:Fcd}

Throughout this section, let $F = F_{c,d} = 1/Q_{c,d} = 1 / (1 - e_1 + c e_d)$
and recall that $c_* = (d-1)^{d-1}$.
Lemma~\ref{lem:GWS} implies that for $Q \in \Md$, in the diagonal 
direction, one may find diagonal minimal points.  For $F_{c,d}$, 
things are even simpler: all critical points for diagonal asymptotics
are diagonal points.
\begin{lemma} \label{lem:Fcd}
Let $F_{c,d} = 1/Q_{c,d}$.
If $\zz \in \crit (1, \ldots , 1)$ then $z_i = z_j$ for all $1 \leq i, j \leq d$.
\end{lemma}

\begin{proof} From $Q = Q_{c,d} = 1 - e_1 + c e_d$ we see that
$(\grad_{\log} Q)_j = -z_j - c e_d$ and hence that $(\grad_{\log} Q)_i
= (\grad_{\log} Q)_j$ if and only if $z_i = z_j$.  
\end{proof}

\begin{prop}[Smoothness of $F_{c,d}$ for $c \neq c_*$] \label{pr:d-1}
Let $F_{c,d} = 1/Q_{c,d}$.  If $c \neq c_* $ 
then $\sing$ is smooth.  If $c = c_*$ then $\sing$ fails to be smooth 
at the single point $\zz_* = (1/(d-1), \ldots , 1/(d-1))$.  When $c = c_*$, 
the singularity at $\zz_*$ has tangent cone $e_2$.
\end{prop}

\begin{proof}
Checking smoothness of $\sing$ we observe that for $d$ fixed and
$c$ and $x_1, \ldots , x_d$ variable, vanishing of the gradient
of $Q_{c,d}$ with respect to the $x$ variables implies $x_j = c e_d$
for all $j$.  This common value, $x$, cannot be zero, hence $x_j \equiv x$
and $c = x^{1-d}$.  Vanishing of $Q_{c,d}$ then implies vanishing of $1 - dx + x$,
hence $x = 1/(d-1)$ and $c = c_*$.  This proves the first 
two statements.  Setting $c = c_*$ and $x_j = 1/(d-1) + y_j$ centers
$Q_{c_*,d}$ at the singularity and produces a leading term of 
$(d-1) e_2 (\yy)$, proving the third statement.
\end{proof}

\subsection{Proof of Theorems~\protect{\ref{th:asym pos}} 
and~\protect{\ref{th:asm Fcd}} in the case $c < c_*$}

When $c \leq 0$, the denominator of $F_{c,d}$ is one minus the sum 
of positive monomials, which leaves no doubt as to positivity.  
Assume, therefore, that $0 < c < c_*$.  
Apply Lemma~\ref{lem:GWS} to see that if $x$ is a
minimum modulus zero of $\delta^Q := Q_{c,d} (x, \ldots ,x)$ 
then $(x, \ldots , x)$ is a minimal point for $F_{c,d}$ in the 
diagonal direction.  Apply Lemma~\ref{lem:Fcd} to conclude that
the set $E$ in Theorem~\ref{th:smooth} of minimal critical points
on $\TT (|x|, \ldots , |x|)$ consists only of points $(y, \ldots , y)$
such that $y$ is a root of $\delta^Q$.  By part~\eqref{item:1} of 
Corollary~\ref{cor:asym pos}, it suffices to check that 
$\delta^Q = 1 - dx + c x^d$ has a unique minimal modulus root 
$\rho$ and that $\rho \in \R^+$.  Thus, the conclusion follows 
from the following proposition.

\begin{prop} \label{pr:c < c_*}
For $c \in (0,c_*)$, the polynomial $\delta^Q = 1 - dx + c x^d$ has
a root $\rho \in \left [ \frac{1}{d} , \frac{1}{d-1} \right]$
which is the unique root of $\delta^Q$ of modulus less than $1/(d-1)$.
\end{prop}

\begin{proof} 
Checking signs we find that $\delta^Q (1/d) = 
c d^{-d} > 0$ while $\delta^Q (1/(d-1)) = - (d-1)^{-1} + c (d-1)^{-d}
< - (d-1)^{-1} + c_* (d-1)^{-d} = 0$, therefore there is at
least one root, call it $\rho$, of $\delta^Q$ in the interval 
$[1/d, 1/(d-1)]$.  On the other hand, when $|z| = 1/(d-1)$, we see 
that $|d z| \geq |1 + c z^d|$ and therefore, by applying Rouch{\'e}'s 
theorem to the functions $-dz$ and $1 + c z^d$, we see that
$\delta^Q$ has as many zeros on $|z| < 1/(d-1)$ as does $-dz$:
precisely one root, $\rho$.
\end{proof}

\subsection{Proof of Theorems~\protect{\ref{th:asym pos}} 
and~\protect{\ref{th:asm Fcd}} in the case $c > c_*$}

Again, by Lemmas~\ref{lem:GWS} and~\ref{lem:Fcd}, we may apply
part~\eqref{item:3} of Corollary~\ref{cor:asym pos} to the set $E$ of points 
$(y, \ldots , y)$ for all minimal modulus roots $y$ of $\delta^Q$.
The result then reduces to the following proposition.
\begin{prop}
For $c > c_*$, the set of minimal modulus roots of the polynomial 
$\delta^Q = 1 - dx + c x^d$ contains no point whose $d^{th}$ power 
is real and positive.
\end{prop}

\begin{proof}
First, if $z^d$ is real then the imaginary part
of $\delta^Q (z)$ is equal to the imaginary part of $-dz$, hence 
any root $z$ of $\delta^Q$ with $z^d$ real is itself real.

Next we check that $\delta^Q$ has no positive real roots.  
Differentiating $\delta^Q (x)$ with respect to $x$ gives 
the increasing function $d (-1 + cx^{d-1})$ with a
unique zero at $c^{-1/(d-1)}$.  This gives the location of the
minimum of $\delta^Q$ on $\R^+$, where the function value is
$1 - d c^{-1/(d-1)} + c^{1 - d/(d-1)} = 1 - (d-1) / c^{1/(d-1)}$ 
which is positive because $c > (d-1)^{d-1}$.  

If $d$ is even, $\delta^Q$ clearly has no negative real roots, 
hence no real roots at all, finishing the proof in this case.
If $d$ is odd $\delta^Q$ will have a negative real root $u$, however 
because $d$ is odd, the product of the coordinates of 
$(u, \ldots , u)$ is $u^d < 0$. 
\end{proof}

We conjecture that the roots of minimal modulus when $c>c_*$ are 
always a complex conjugate pair, however this determination does not
affect our positivity results.

\subsection{Proof of Theorem~\protect{\ref{th:d-1}}}
When $c < c_*$ we have seen that there is a single real minimal point 
$(\rho_c, \ldots , \rho_c)$ in the diagonal direction and that
$\rho_c \uparrow 1/(d-1)$ as $c \uparrow c_*^-$.  
The limit from below in Theorem~\ref{th:d-1} then follows 
directly from Theorem~\ref{th:asm Fcd}.

For the limit from above, 
it suffices to show that in the diagonal direction, 
for $c$ sufficiently close to $c_*$ and greater, $E$ consists 
of a single diagonal complex conjugate pair $(\zeta_c, \ldots , \zeta_c)$
and $(\overline{\zeta_c}, \ldots , \overline{\zeta_c})$, and that
$\overline{\zeta_c} \to 1/(d-1)$ as $c_* \downarrow c$.
First, we check that at $c = c_*$ the unique minimum modulus
root of $\delta^Q$ is the doubled root at $1/(d-1)$.  For $c = c_*$,
the first and third terms of $\delta^Q = 1 - dz + c_* z^d$ have modulus~1
and $1/(d-1)$ when $|z|=1/(d-1)$, respectively, summing to the modulus of 
the middle term; therefore if $\delta^Q(z) = 0$ and $|z| = 1/(d-1)$ then 
the third term is positive real.  But then the second term must be positive 
real too, hence the unique solution of modulus at most $1/(d-1)$ is 
$z = 1/(d-1)$.  A quick computation shows the multiplicity to be precisely~2.  
We know that for $c > c_*$ there are no real roots.  Therefore, as $c$ 
increases from $c_*$, the minimum modulus doubled root splits into two 
conjugate roots, which, in a neighborhood of $c_*$, are still the only 
minimum modulus roots.  

\section{Lacuna computations} \label{sec:lacuna}

Theorem~\ref{th:lacuna} is the subject of forthcoming work~\cite{BMP-lacuna}.
Theorem~\ref{th:d-1 critical} follows immediately, 
with the specifications: $d \geq 4$ and even, $c = c_*$, $k=1$,
$P = 1$, $Q = Q_{c,d}$, $\zz_* = (1/d, \ldots , 1/d)$,
$\rhat = (1, \ldots , 1)$, $B$ is the component of the
complement of the amoeba of $Q$ containing $(a, \ldots , a)$
for $a < - \log d$, $\xx_* = (-\log d, \ldots , - \log d)$,
$\yy_* = {\bf 0}$ and $\nbd$ taken to be the diagonal. 
Proposition~\ref{pr:d-1} guarantees the correct shape for 
the tangent cone to $Q$ at $\zz_*$.

\begin{theorem} \label{th:lacuna}
Suppose $F = P/Q^k$ with $P$ a holomorphic function and $Q$ a real 
Laurent polynomial.  Fix $\rhat \in \RP^d$, let $B$ be a component 
of the complement of the amoeba of $Q$, let $\sum_\rr a_\rr \zz^\rr$ 
be the Laurent expansion for $F$ convergent for $\zz = \exp (\xx + i \yy)$ 
and $\xx \in B$.  Let $\xx_* \in \partial B$ be a maximizing point
for $\rr \cdot \xx$ on $\partial B$.  Assume that $\sing$ has a 
unique singularity $\zz_* = \exp (\xx_* + i \yy_*)$, and that
the tangent cone of $Q$ at $\zz$ transforms by a real linear map
to $z_d^2 - \sum_{j=1}^{d-1} z_j^2$.  Let $\nbd$ be any closed cone
such that $\xx_*$ maximizes $\rr \cdot \xx$ for all $\rr \in \nbd$.

If $d > 2k$ is even then there is an $\ee > 0$ and a chain 
$\Gamma$ contained in the set $\sing_{\-\ee} := \{ \zz \in \sing : 
|\zz^{-\rr}| \leq \exp ( - \rr \cdot \xx_* - \ee |\rr|)$ such that 
\begin{equation} \label{eq:int chain}
a_\rr = \int_\Gamma \zz^{-\rr} \frac{P}{Q^k} \frac{d\zz}{\zz} \, .
\end{equation}
In other words, the chain of integration can be slipped below 
the height of the singular point. 
\end{theorem}

\begin{proof}[Sketch of proof:]
Expand the torus $\TT$ of integration to $\zz_*$ and just beyond.
The integral~\eqref{eq:cauchy} turns into a residue integral 
over an intersection cycle swept out by the expanding torus;
see, e.g.~\cite[Appendix~A.4]{PW-book}.  For small perturbations 
$Q_\ee$ of $Q$, the residue cycle is the union of a sphere surrounding $\zz_*$
and a hyperboloid intersecting the sphere.  As $Q_\ee \to Q$, this 
cycle may be deformed so that the sphere shrinks to a point while the
hyperboloid's neck also constricts to a point.  The hyperboloid may then
be folded back on itself so that in a neighborhood of $\zz_*$, the
chain vanishes, leaving a chain $\Gamma$ supported below the height
of $\zz_*$.
\end{proof}  

\appendix
\section{Appendix A: Maple Code}

Maple worksheets going through the calculations discussed above can be found at 
\url{https://github.com/smelczer/SymmetricRationalFunctionsAofA} ;
we include the main component of those worksheets, code giving dominant smooth asymptotics, 
here for archival purposes.

\begin{align*}
&\text{smoothASM := proc($G,H,\text{vars},\text{pt}$)}\\
&\qquad \text{local $N,i,j,M,\text{HES},C,U,\text{lambda},\text{sbs}$:} \\
&\qquad N := \text{nops}(\text{vars}): \\[+2mm]
&\qquad\text{\# Get the Hessian determinant of the phase implicitly}\\
&\qquad\text{for $i$ from 1 to $N$ do for $j$ from 1 to $N$ do} \\
&\qquad\qquad U[i,j] := \text{vars}[i] \cdot \text{vars}[j] \cdot \text{diff}(Q,\text{vars}[i],\text{vars}[j]): \\
&\qquad\text{od: od:} \\
&\qquad \text{lambda} := x \cdot \text{diff}(Q,x): \\
&\qquad\text{for $i$ from 1 to $N-1$ do for $j$ from 1 to $N-1$ do} \\
&\qquad\qquad \text{if $i <> j$ then } M[i,j] := 1 + 1/\text{lambda} \cdot (U[i,j]-U[i,N]-U[j,N]+U[N,N]):  \\
&\qquad\qquad \text{else } M[i,j] := 2 + 1/\text{lambda} \cdot (U[i,i]-2 \cdot U[i,N]+U[N,N]): \\
&\qquad\quad \text{fi:}  \\
&\qquad\text{od: od:} \\
&\qquad\text{HES} := \text{LinearAlgebra[Determinant](Matrix([seq([seq($M[i,j], i=1..N-1)$], $j=1..N-1$)]))}: \\[+2mm]
&\qquad C := \text{simplify}(-G / \text{vars}[-1] / \text{diff}(H,\text{vars}[-1]) \cdot \text{HES}{\mathchar"5E}(-1/2) \cdot (2\cdot\text{Pi}){\mathchar"5E}((1-N)/2)); \\
&\qquad \text{sbs := seq}(\text{vars}[j]=\text{pt}[j],j=1..N): \\
&\qquad \text{return eval}(1/\text{mul}(j,j=\text{pt})){\mathchar"5E}n \cdot n{\mathchar"5E}((1-N)/2) \cdot \text{eval}(\text{subs}(\text{sbs}, C)): \\
&\text{end:}
\end{align*}


\bibliography{Posbib}

\end{document}